\documentclass[10pt]{article}

\usepackage{amssymb}
\usepackage{amsfonts}
\usepackage{amsmath}
\usepackage{amssymb}
\usepackage{graphicx}
\usepackage{caption}
\usepackage{float}
\usepackage[utf8]{inputenc}
\usepackage[english]{babel}
\usepackage{amsthm}

\newtheorem{theorem}{Theorem}
\newtheorem{lemma}[theorem]{Lemma}

\newtheorem{remark}{Remark}
\newtheorem{definition}{Definition}
\newtheorem{proposition}{Proposition}

\begin{document}
			
		\title{Safe 3-coloring of graphs}
		

		\author{Tanja Vojkovi\'c (1) \footnote{Corresponding author: tanja@pmfst.hr}, Damir Vuki\v{c}evi\'c (1)}

			
		\maketitle
		
		\noindent (1) Department of Mathematics, Faculty of Science, Split, Croatia\\

		\begin{abstract}
			The applications of graph coloring are diverse and many so lots of new types of coloring are being proposed and explored. Here we define a safe $k$-coloring, motivated by the application of coloring to secret sharing. Secret sharing is a way of securing a secret from a number of attackers by dividing it into parts and then distributing those parts to some persons, represented here by graph vertices. Parts of the secret are represented by colors which are then assigned to the vertices under certain conditions, making a coloring safe if a predetermined number of attackers cannot read the whole secret, nor disable the rest of the group from doing so. We observe a fixed number of colors, namely 3, and analyze what kind of graphs have a safe $3$-coloring.
		
		\end{abstract}
	
		
		05C82, 05C15, 68R10, 94A62 
		%%MSC codes here, in the form: \MSC code \sep code

\section{Introduction and motivation}
\label{}

Graph colorings are a well known subject in graph theory. From the early days and the Four Color Theorem to many applications in scheduling, frequency allocation and timetabling \cite{appel,parlados,marx}. Simply put, a graph coloring is a function which assigns a color to every vertex or edge of the graph, hence vertex colorings and edge colorings. A coloring is proper if two adjacent vertices (edges) are not assigned the same color. Throughout the years, many different variations of coloring have been presented and studied, with many different conditions, for instance rainbow and anti-rainbow colorings of planar graphs, star colorings, list colorings, or multicolorings where a set of colors is assigned instead of a single color \cite{borodin,czap,fertin,halldorson}. In these colorings, different problems have been presented. Mostly the goal is to determine a minimal number of colors to color the graph properly or respecting some special conditions, but other goals have also been explored, like analyzing families of graphs that are colorable in a specific way, or developing efficient algorithms for specific coloring \cite{bollobas2,borodin2,blum}. In this paper we will present a variation of graph vertex coloring, motivated by the problem of securing a secret. We name it safe coloring. The idea is that some secret code or message is not safe enough if kept in one place, so it is divided into pieces and those pieces are distributed to the actors of some group. This is a well known method of secret sharing in cryptography \cite{shamir}. Usually the assumption is that some of the actors are corrupted, they are "the attackers", which behave in a certain way to steal the secret of prevent the rest of the group from reading it. In our considerations, a group is represented by a graph, and each piece of the secret corresponds to one color which is then distributed to the vertices. There have been some applications of graph coloring in secret sharing and network security, some of them are given in the following references \cite{desmedt,pal}.\\
In our previous paper, Multicoloring of Graphs to Secure a Secret \cite{nas}, we also started with a problem of dividing a secret into parts and distributing those parts to graph vertices. However, there we assumed the behavior of attackers in such a way that in order to secure a secret we observed multicolorings instead of colorings. There we defined a new kind of multicoloring, a highly $a$-resistant vertex $k$-multicoloring, and we analyzed minimal number of colors for such a coloring to exist, for different number of attackers, $a$. \\
Here our approach is different. The motivation of securing a secret against a number of corrupted vertices (the attackers) yields conditions on the coloring which prompt us to define a safe vertex coloring. In this paper we will restrict our observations to a fixed number of colors, namely $3$, and analyze the family of graphs that have a safe coloring with $3$ colors. 
The conditions for safe coloring follow from the assumption that the secret is safe if the group of attackers didn't manage to read the whole secret, i.e. collect all the pieces, and further, that they didn't disable the rest of the group from reading the secret. We assume that the attackers leave the group at some point (the attacker vertices are removed from graph), and the group can still read the secret if there is a component of the remaining graph that has all the pieces.
The main part of the paper consists of three sections. First we formulate the problem in a mathematical way, define safe coloring and present some additional conditions and restrictions under which we proceed. In the section Main results we describe and prove which families of graphs have a safe $3$-coloring, and in the section Additional results we give some remarks about time complexity of algorithms  which check if a given graph is safely $3$-colorable and a few minor observations.

\section{The definition of safe coloring}

Throughout the paper we will use standard definitions and notation of graph theory \cite{bollobas}.  
Graph $k$-coloring is a function $\phi :V(G)\rightarrow \{1,2,...,k\}$ which colors every vertex of a graph in one of $k$ colors. The coloring is proper if no adjacent vertices receive the same color. With $G\backslash A$, where $A\subset V(G)$, we denote a graph obtainted from graph $G$ by removing all vertices from $A$ and their incident edges. Now we formally introduce a concept of safe $k$-coloring.

\begin{definition}
	An \textbf{$a$-safe $k$-coloring} is a function $\phi :V(G)\rightarrow \{1,2,...,k\}$ such that for each subset $A\subset V(G)$, where $|A|=a$ it holds
	\begin{enumerate}
		\item 	${\displaystyle\bigcup\limits_{u\in A}}\phi(u)\neq\{1,...,k\}$;
		\item There is a component $H$ of graph $G\backslash A$ such that
		$${\displaystyle\bigcup\limits_{u\in V(H)}}\phi(u)=\{1,...,k\}\text{.}$$
		If some $a$-safe $k$-coloring exist for graph $G$ we say that $G$ is \textbf{$a$-safely $k$-colorable}.
	\end{enumerate}
\end{definition}

From condition $1.$ of the definition, it is easy to see that for a graph to be $a$-safely $k$-colorable it must hold $a\leq k-1$.
When $a=k-1$ we will call an $a$-safe $k$-coloring simply a \textbf{safe $k$-coloring}, and we will observe safe $k$-colorings in this paper.
Note that safe $k$-coloring doesn't need to be proper.

Our goal is to answer the question: What are the graphs that allow a safe $k$-coloring? In this paper we will restrict our efforts to $k=3$ and determine and prove which graphs admit a safe $3$-coloring.
We will observe only graphs with minimal degree at least $3$, motivated by the definition of safe coloring. Namely, if we demand that a component with all colors must exist in a graph with some vertices removed, then it is reasonable to assume a lower bound for minimal degree, as to make the number of "small" components in graph $G\backslash A$ as little as possible.\\

First, let us make an observation that a graph that admits a safe $3$-coloring must have at least $9$ vertices. In contrary, if it has at most $8$ vertices, at least one of the colors will appear at most $2$ times and then with the choice of those vertices in subset $A$ the defining conditions don't hold.

We will prove that all graphs $G$, with $|V(G)|\geq 9$ and $\delta (G)\geq 3$ have a safe $3$-coloring, with two
exceptions, a double windmill with adjacent centers and a double windmill with non-adjacent centers. Let us define those graphs.

\begin{definition}
	A \textbf{double windmill with adjacent centers}, $\overline{DW_{l}}$, $l\geq 1$ is a graph
	which consists of $l$ graphs $K_{2}$ and two central vertices which are
	adjacent to all the vertices from all $K_{2}$ graphs, and to each other. If
	central vertices are not adjacent we call it a \textbf{double windmill with
	non-adjacent centers} and denote it by $DW_{l}$. (Figure \ref{fig:DW}.)
\end{definition}

	\begin{figure}[h]
	\centering\includegraphics[scale=0.7]{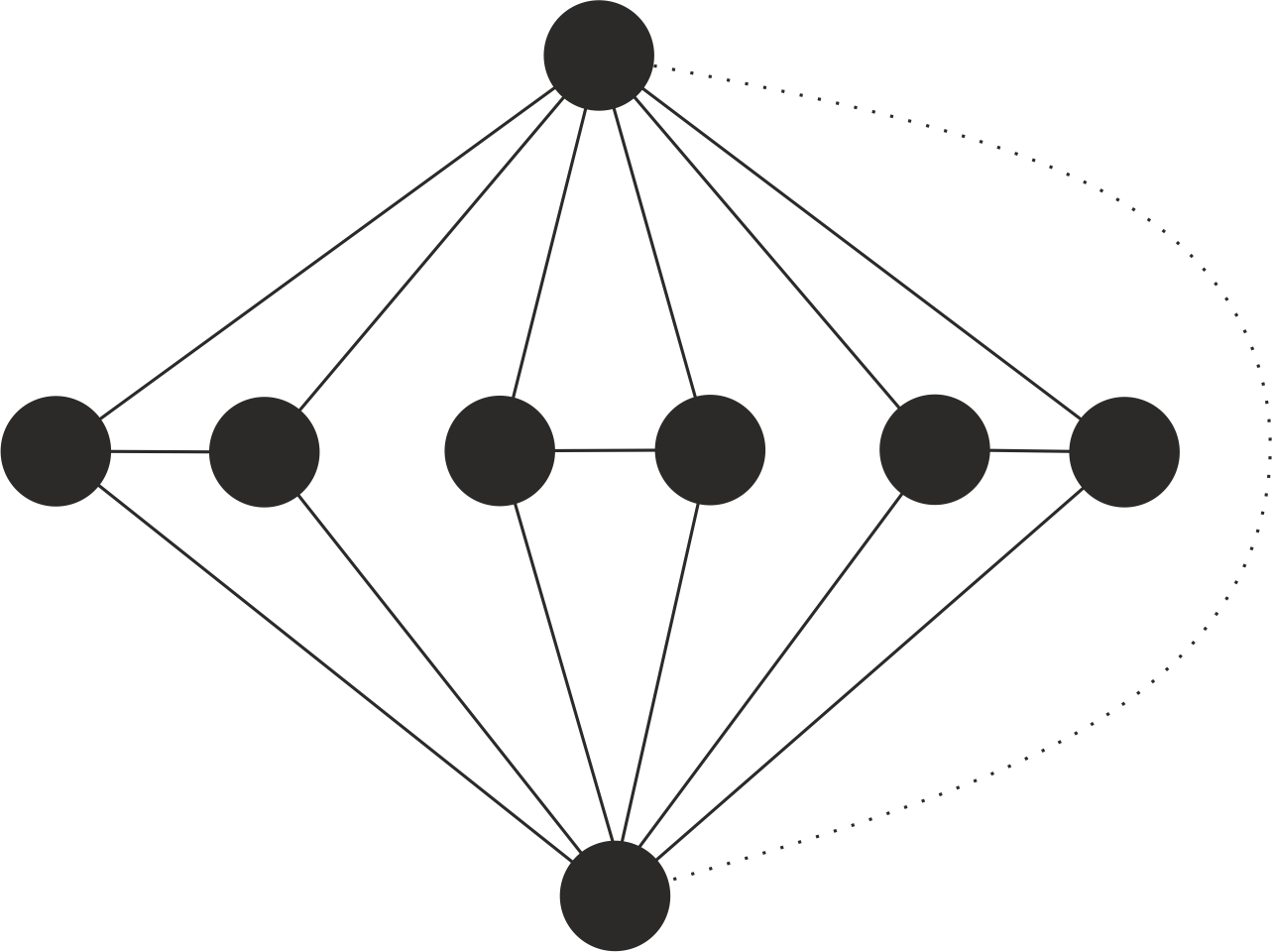}
	\caption{A double windmill, $l=3$}
	\label{fig:DW}
\end{figure}

By \textbf{double windmill} we will mean any of two kinds of double windmills defined.
It is easy to see that both a double windmill with adjacent centers and a double windmill with non-adjacent centers don't have a safe $3$-coloring, for if we choose $A=\{c_{1},c_{2}\}$, where $c_{1}$ and $c_{2}$ are the two centers, all the components in $G\backslash A$ will have $2$ vertices and therefore none of them has all $3$ colors.

A structure we will widely use in our considerations is that of a \textbf{connected
triplet}. It means $3$ vertices connected by a path, the third, triangle
forming edge, may or may not exist. We will denote a connected triplet by $
T_{3}$. By \textbf{independent triples} we assume the triplets that share no vertices. 
For cycles of length $3$, $4$, or $5$, we will use standard
denotation $C_{3}$, $C_{4}$ and $C_{5}$, and for a path of length $n$, $
P_{n} $. Also, whenever possible we will use $\delta $ instead of $\delta
(G) $ and $n$ instead of $|V(G)|$. 

\section{Main results}
Now let us begin with the claims.

\begin{proposition}
	\label{prop1}A graph $G$ which contains three independent triplets is safely $3$-colorable.
\end{proposition}

\begin{proof}
	The proof is quite obvious, since it is enough to find a coloring which assigns
	colors $1$, $2$ and $3$ to three different vertices in each of three
	triplets, and by removing any two vertices at least one triplet remains
	intact.
\end{proof}

\begin{lemma}
	\label{lm1}A graph $G$ with $n\geq 9$ vertices and $\delta \geq 3$ which contains a
	double windmill as a subgraph is either safely $3$-colorable or a double windmill.
\end{lemma}

\begin{proof}
	Let $G$ be a graph with $n\geq 9$ and $\delta \geq 3$ which contains a
	double windmill as a subgraph. Let $D$ be the largest double windmill in $G$. 
	If there are no vertices in $G\backslash D$, the claim is proven, so let
	us assume there exist a vertex $u\in V(G)\backslash V(D)$. It holds $%
	d(u)\geq 3$, so let us observe neighbors of $u$. We distinguish two cases.
	
	1) $u$ is adjecent to two central vertices of $D$.
	
	Than $u$ must have at least one more neighbor. If that neighbor is any
	vertex in $D$ different from the central vertices than $G$ contains three
	independant triplets, and is therefore safely $3$-colorable. And if the third neighbor
	of $u$ is a vertex $v\in V(G)\backslash V(D)$ than we observe neighbors of $%
	v $. If $v$ has any more neighbors not contained in $D$, the three triplets
	are again easily seen, and if $v$ is adjacent only to $u$ and the two
	central vertics of $D$ than we have obtained a larger windmill which is a
	contradiction.
	
	2) $u$ is not adjecent to both central vertices of $D$. Than $u$ must have
	at least two more neighbors. In all the possible cases, of those neighbors
	be in $D$ or not, the triplets are easily found.
\end{proof}

Now we present the central claim of the paper, which will be proven through several claims.
\begin{theorem}
	Graph $G$ with $\delta \geq 3$ is safely $3$-colorable if at least one of the
	following stands:
	
	i) $G$ has at least three components;
	
	ii) $G$ has two components with at least $6$ vertices each;
	
	iii) $G$ has at least one component with at least $9$ vertices which is
	different from a double windmill.

\end{theorem}

First, let us consider connected graphs and prove:

\begin{theorem}
	A connected graph $G$ with $n\geq 9$ vertices and $\delta \geq 3$ is safely $3$-colorable or it
	is a double windmill.
\end{theorem}

\begin{proof}
	Obviously, a connected graph $G$ with $n\geq 9$ and $\delta \geq 3$ contains
	at least one $T_{3}$. Let us prove that it contains at least two independent
	triplets. We have at least one triplet, hence we can denote with $T_{3}$ a triplet
	for which the sum of degrees of its vertices is minimal. We consider two
	cases, with some subcases.
	
	1) $T_{3}$ is not a triangle. Let us denote its vertices by $u$, $v$, $w$; $
	v $ being the middle one. This means that vertices $u$ and $w$ have at least
	two more neighbors not contained in $T_{3}$, and $v$ has at least one.
	
	1.1) If $u$ has $x_{1}$ and $x_{2}$ for neighbors, and $w$ has $y_{1}$ and $%
	y_{2}$ and all those neighbors are different vertices then the two triplets
	are $x_{1}ux_{2}$ and $y_{1}wy_{2}$.
	
	1.2.) If $u$ has neighbors $x_{1}$ and $x_{2}$ and $w$ has neighbors $x_{1}$
	and $y_{1}$ then the two triplets are $x_{1}ux_{2}$ and $vwy_{1}$.
	
	1.3.) If both $u$ and $w$ have neighbors $x_{1}$ and $x_{2}$ then we
	consider neighbors of $v$.
	
	1.3.1) If $x_{1}$ and $x_{2}$ are adjacent, we have a cycle $uvwx_{2}x_{1}$.
	Since $G$ is connected, there is a vertex $x_{3}$ adjacent to one of these
	vertices and we have two independent triplets.
	
	1.3.2) Suppose that $x_{1}$ and $x_{2}$ are not adjacent. If $v$ is also
	adjacent to $x_{1}$ and not $x_{2}$ then $x_{2}$ has at least one more
	neighbor, $y$. Now, $yx_{2}w$ and $ux_{1}v$ are two triplets. The same holds
	if $v$ is adjacent with $x_{2}$ and not $x_{1}$. If $v$ is adjacent to both $%
	x_{1}$ and $x_{2}$ then $x_{1}$ or $x_{2}$ must have at least one more
	neighbor, since $n\geq 9$, and $uvw$ is a triplet with minimal sum of vertex
	degrees. With that new neighbor, the triplets are formed as in previous.
	Third option is if $v$ is adjacent to neither $x_{1}$ nor $x_{2}$, but
	instead, a new neighbor, $y$, but in that case it is easily seen that the
	triplets for example are $yvw$, $x_{1}ux_{2}$.
	
	2) $T_{3}$ is a triangle. Let us again denote its vertices by $u$, $v$, $w$.
	Now each of $u$, $v$, $w$ has at least one more neighbor.
	
	2.1) If that is the same one, a vertex $x$, then we have a $K_{4}$ subgraph,
	and since $n\geq 9$, $x$ must have more neighbors. ($x$ has more or equal
	neighbors than $u$, $v$ and $w$ because of the vertex degree sum). Either $x$
	has one more neighbor, $y$, and then $y$ has more neighbors, or $x$ has two
	or more neighbors, but in both cases the existence of two independent
	triplets can be easily seen.
	
	2.2) If any two vertices of $u$, $v$, and $w$ are neighbors with the same
	vertex $x$, let us assume $u$ and $v$, and the third one, $w$, is neighbor
	with vertex $y\neq x$. Then $y$ has at least one more neighbor, $z$, in
	which case $zyw$ and $uvx$ are the triplets, or is adjacent with $x$ and $u$
	or $v$. In this case, again because of the minimal sum of degrees of $uvw$, $%
	x$ or $y$ must have at least one more neighbor, and then the triplets are
	easily seen.
	
	The proof of the Theorem is divided in several claims.
	
	\textbf{CLAIM 1}. A connected graph $G$ with $n\geq 9$ and $\delta \geq 3$
	which contains independent $C_{5}$ and $T_{3}$ is safely $3$-colorable.
	
	\textit{Proof of Claim 1.} In $C_{5}$ and $T_{3}$ there are $8$ vertices, so $G$
	contains at least another vertex, $u$. Let us denote the vertices in $T_{3}$
	by $x_{1}$, $x_{2}$ and $x_{3}$; $x_{2}$ being the middle one. (Figure \ref
	{fig:c5t3}.) 
	
		\begin{figure}[h]
		\centering\includegraphics[scale=0.4]{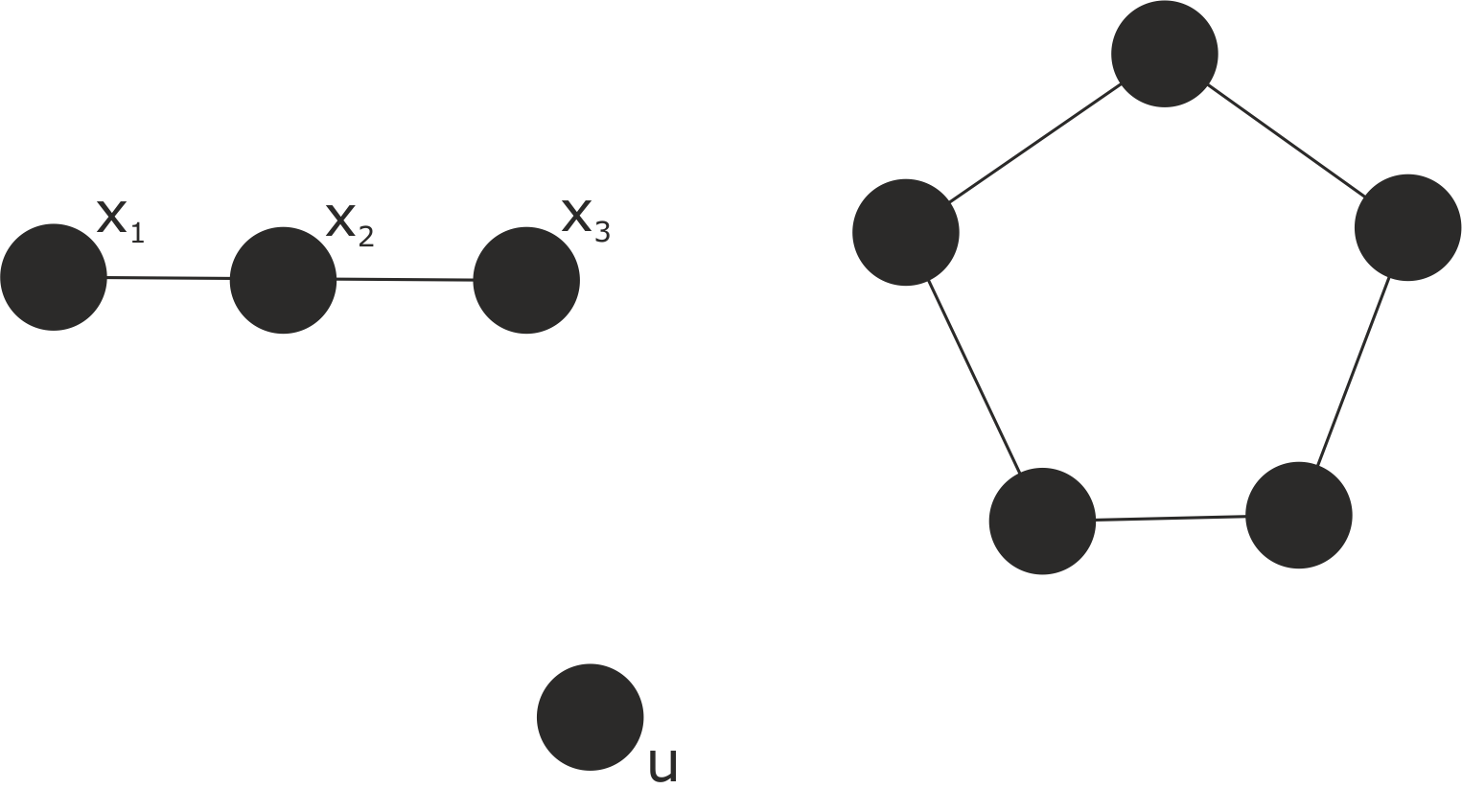}
		\caption{Independent $C_{5}$ and $T_{3}$}
		\label{fig:c5t3}
	\end{figure}

	If $u$ is adjacent with any vertex in $%
	C_{5}$, then there we have a path $P_{6}$, and with $T_{3}$ we have three
	independent triplets so $G$ is safely $3$-colorable by Proposition \ref{prop1}. If $u$ has
	two neighbors not contained in $T_{3}$ then obviously we have three
	triplets, so let us assume $u$ has two neighbors in $T_{3}$, and at least one of them must be
	the end one. So let $u$ be adjacent to $x_{1}$. Since $u$ is not adjacent to 
	$C_{5}$, and $G$ is a connected graph, $C_{5}$ must be adjacent to $T_{3}$ by
	some path. If there are additional vertices on that path, we will have three
	triplets, so let us assume $C_{5}$ is adjacent with $T_{3}$ by an edge. We
	consider three cases:
	
	1) $C_{5}$ is adjacent to $x_{1}$. Since $u$ is adjacent to $x_{1}$ and
	either $x_{2}$ or $x_{3}$, we have a triplet $ux_{2}x_{3}$, and $x_{1}$ with 
	$C_{5}$ forms a path $P_{6}$ and hence two additional triplets.
	
	2) $C_{5}$ is adjacent to $x_{3}$. In this case we have a triplet $%
	ux_{1}x_{2}$ and a path $P_{6}$.
	
	3) $C_{5}$ is adjacent to $x_{2}$. Now $C_{5}$ and $x_{2}$ form two
	triplets, and in addition to $x_{1}$, $u$ is adjacent either to $x_{3}$, so $%
	x_{1}ux_{3}$ is a third triplet, or $u$ is adjacent to $x_{2}$ and some \
	other vertex, $w$, not contained in $C_{5}$ nor $T_{3}$. But now $x_{1}uw$
	is the third triplet. $\square $
	
	\textbf{CLAIM\ 2}.\ A connected graph $G$ with $n\geq 9$ and $\delta \geq 3$
	which contains independent $C_{4}$ and $C_{3}$ is safely $3$-colorable.
	
	\textit{Proof of Claim 2.} Besides $C_{4}$ and $C_{3}$, $G$ must have at least two
	more vertices. Let us denote them by $u$ and $v$, and let us denote the
	vertices in $C_{4}$ by $x_{1},...,x_{4}$, and in $C_{3}$ by $y_{1}$, $y_{2}$, $%
	y_{3}$. (Figure \ref{fig:c4c3}.)
	
		\begin{figure}[h]
		\centering\includegraphics[scale=0.4]{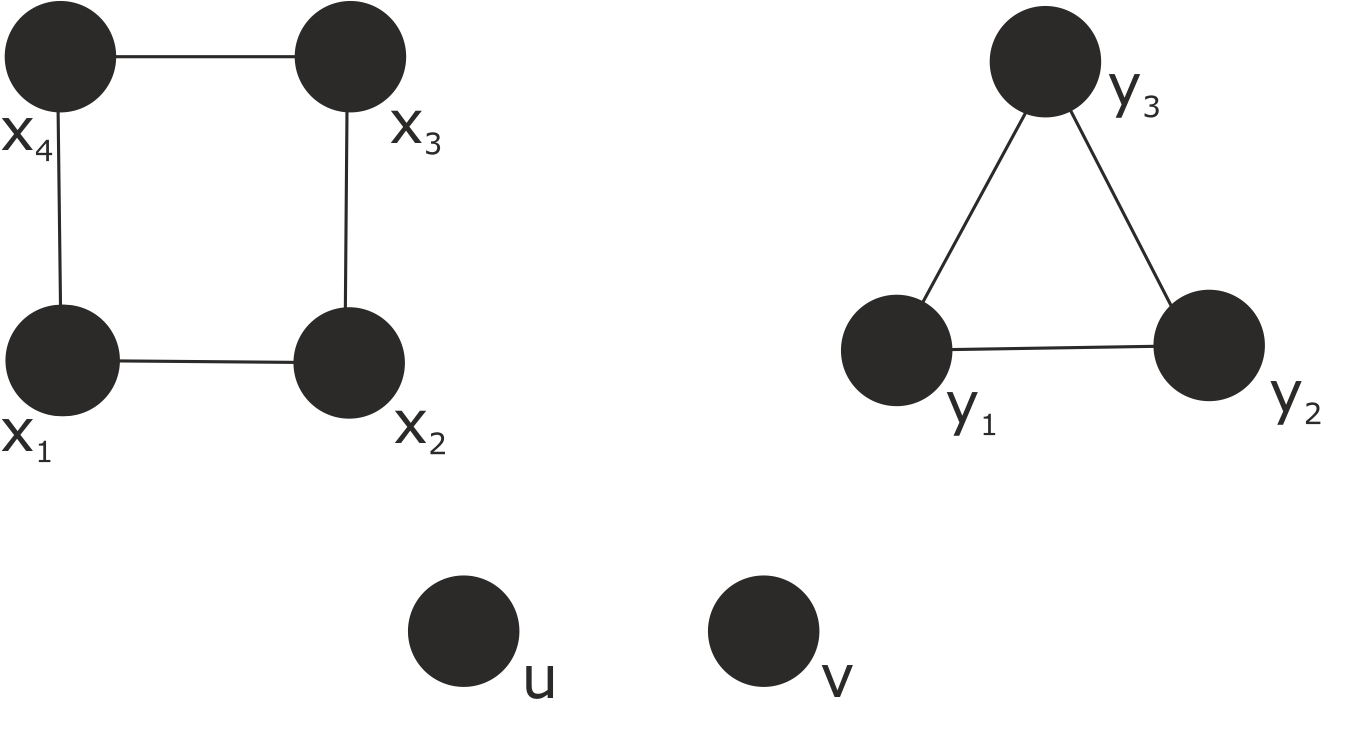}
		\caption{Independent $C_{4}$ and $C_{3}$}
		\label{fig:c4c3}
	\end{figure}

	If $u$ or $v$ have two neighbors not contained in $C_{4}$ or $C_{3}$ we have three
	independant triplets, so let us assume this is not the case. This means that
	both $u$ and $v$ have at least two neighbors each in $C_{4}\cup C_{3}$. If $%
	u $ or $v$ are adjacent to two neighbor vertices in $C_{4}$ then together
	they form $C_{5}$, and with $C_{3}$, we have the conditions form Claim 1, so 
	$G$ is safely $3$-colorable. Let us distinguish two cases and their subcases:
	
	1) Let $u$ be adjacent to any vertex in $C_{4}$, say $x_{1}$, and let us
	observe neighbors of $v$.
	
	1.1.) If $v$ is adjacent to $u$ or to $x_{1}$, then $vux_{1}$, $%
	x_{2}x_{3}x_{4}$ and $y_{1}y_{2}y_{3}$ are independent triplets.
	
	1.2.) If $v$ is adjacent to $x_{2}$, $x_{3}$ or $x_{4}$ then in $C_{4}\cup
	\{u\}\cup \{v\}$ we have two triplets.
	
	1.3.) If $v$ is not adjacent to $u$ nor to $C_{4}$ then two neighbors of $v$
	are in $C_{3}$ and in order for $G$ to be connected, $C_{3}$ must be connected
	to $u$ or to $C_{4}$. If $u$ is adjacent to some vertex in $C_{3}$ then the triplets can be
	formed such that one contains $u$ in the middle and $2$ vertices with which $%
	u $ is adjacent, one in $C_{3}$ and one in $C_{4}$. The other triplet
	contains the three remaining vertices of $C_{4}$, and the third contains $v$
	and two remaining vertices of $C_{3}$. If instead there exists an edge
	between $C_{3} $ and $C_{4}$ then one triplet contains one vertex from $%
	C_{3} $ and two from $C_{4}$ and the other two are easily seen.
	
	2) $u$ is adjacent to two vertices in $C_{3}$. If $v$ is adjacent to any of
	the verices in $C_{4}$ then the proof is the same as the case 1.3. If this
	is not the case then $v$ either has two neighbors in $C_{3}$ and a neighbor
	not contained in $C_{4}\cup C_{3}$, in which case the triplets are easily
	seen, or $v$ is adjacent to all three vertices in $C_{3}$. But now there
	again must exist a path of at least one edge between $C_{4}$ and $C_{3}$ and the triplets are
	easily found. $\square $
	
	\textbf{CLAIM\ 3}. A connected graph $G$ with $n\geq 9$ and $\delta \geq 3$
	which contains two independent $C_{3}$ cycles is safely $3$-colorable or it is a
	double windmill.
	
	\textit{Proof of Claim 3.} There are six vertices in two $C_{3}$ cycles, so $G$ must
	contain at least three more vertices. Let us denote them by $u$, $v$ and $w$, and the vertices in cycles by $x_{i}$, $y_{i}$, $i=1,2,3$. (Figure \ref%
	{fig:c3c3}.)
	
	\begin{figure}[h]
		\centering\includegraphics[scale=0.4]{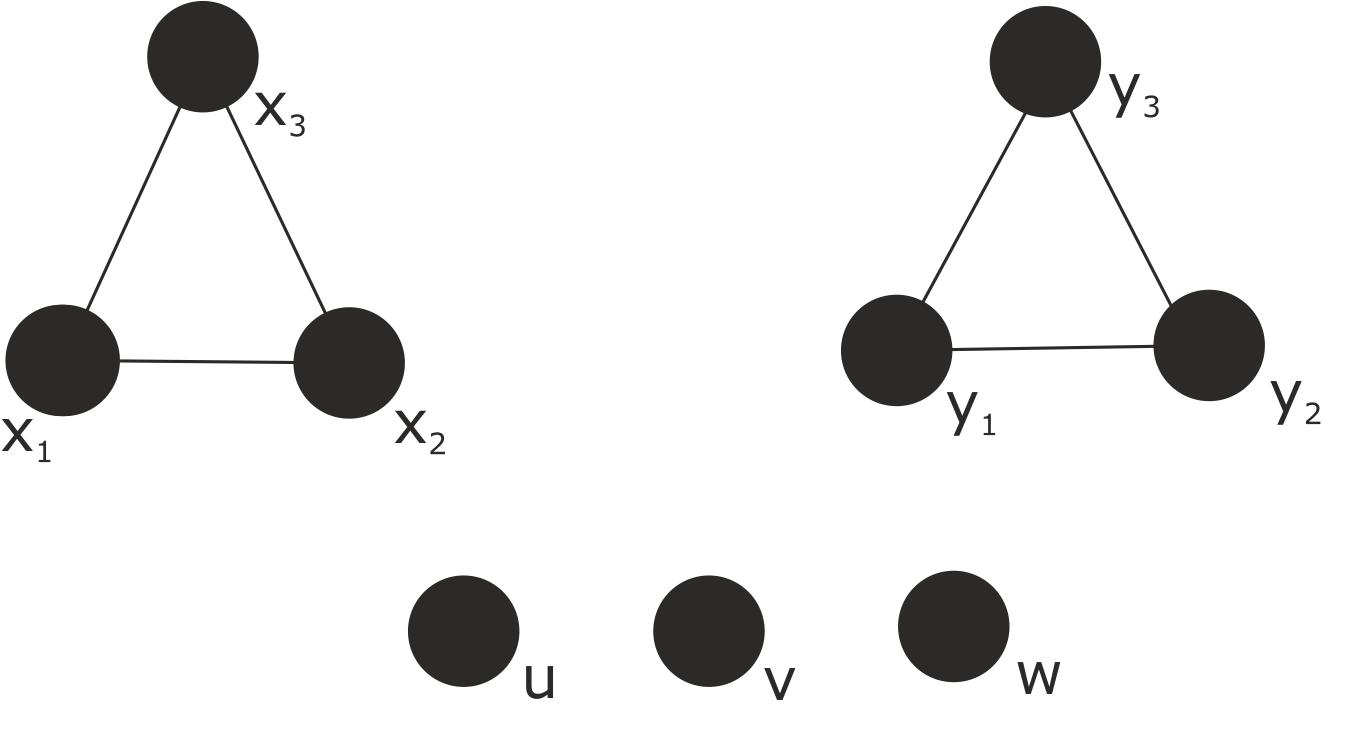}
		\caption{Two independent $C_{3}$ cycles}
		\label{fig:c3c3}
	\end{figure}

	Again, if any vertex of $u$, $v$, $w$
	has two neighbors not contained in $C_{3}\cup C_{3}$ the claim is proven, so
	let us assume that those three vertices have at least two neighbors in $%
	C_{3}\cup C_{3}$. Let us observe also that if any vertex of $u$, $v$, $w$
	has both neighbors in the same $C_{3}$ they form a $C_{4}$ and by Claim 2,
	the Claim 3 is proven. So each of $u$, $v$, $w$ has exactly one neighbor in
	the set $\{x_{1},x_{2},x_{3}\}$ and one in the set $\{y_{1},y_{2},y_{3}\}$.
	Now we consider two cases:
	
	1) Suppose that the joint number of neighbors of $u$, $v$ and $w$ is at least two in
	each of the sets $\{x_{1},x_{2},x_{3}\}$ and $\{y_{1},y_{2},y_{3}\},$ then
	the triplets are easily found (see Figure \ref{fig:c3c3_1}).
	
	\begin{figure}[h]
		\centering\includegraphics[scale=0.4]{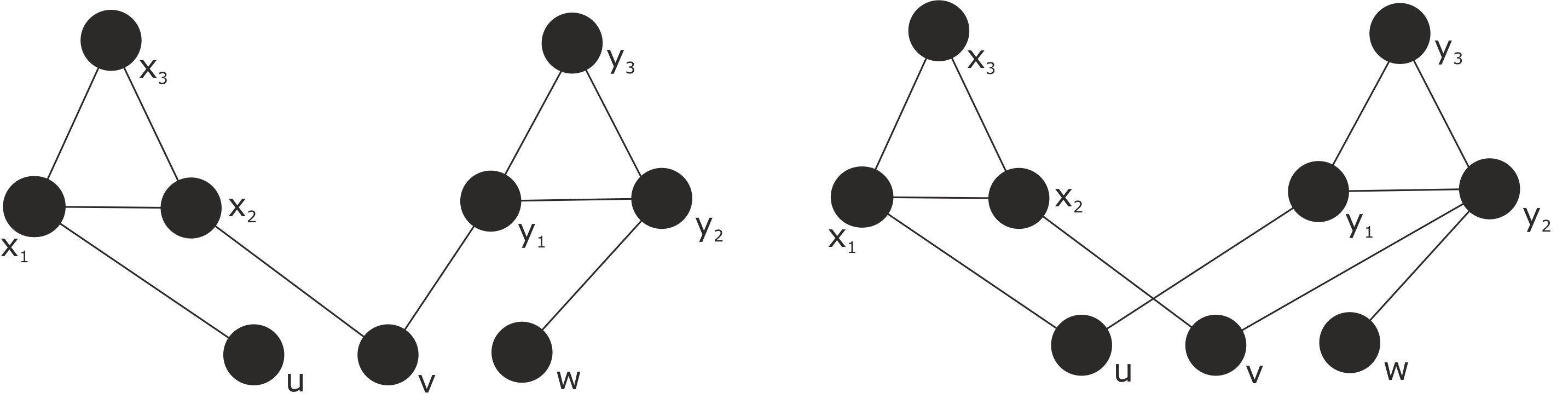}
		\caption{Examples of the first case}
		\label{fig:c3c3_1}
	\end{figure}
	
	2) $u$, $v$ and $w$ have the same neighbor in one of the cycles.
	
	Let us assume it is the vertex $x_{1}$.
	
	2.1.) $u$, $v$ and $w$ all have different neighbors in the set $%
	\{y_{1},y_{2},y_{3}\}$. Obviously $u$, $v$ and $w$ each must have one more
	neighbor. If any of them have a neighbor not contained in $C_{3}\cup C_{3}$
	the triplets can easily be found. (this includes the case when any two of
	them are adjacent by an edge, for then we have a $C_{4}$ and a $C_{3}$ and
	the claim follows from Claim 2.) If any of them have another neighbor in $%
	\{x_{1},x_{2},x_{3}\}$, we are in conditions of the case 1). And if any of
	them have another neighbor in $\{y_{1},y_{2},y_{3}\}$ then we have a $C_{4}$
	and a $C_{3}$ and the claim follows from Claim 2.
	
	2.2.) $u$, $v$ and $w$ have two different neighbors in the set $%
	\{y_{1},y_{2},y_{3}\}$. Without the loss of generality let $u$ and $v$ be
	adjacent to $y_{1}$ and $w$ to $y_{2}$. Now the triplets are $vy_{1}u$, $%
	wy_{2}y_{3}$ and $x_{1}x_{2}x_{3}$.
	
	2.3.) $u$, $v$ and $w$ are all adjacent to the same vertex in the other
	cycle. Without the loss of generality let it be $y_{1}$. First, let us
	observe the vertices $x_{2}$, $x_{3}$, $y_{2}$ and $y_{3}$. If there are any
	edges between them which are not in already observed cycles then the triplets
	are easily seen. So let us assume there aren't any. Now let us look at $u$, $%
	v$ and $w$. They all have two observed neighbors for now. If any of them are adjacent
	to any other vertices in cycles, besides $x_{1}$ and $y_{1}$ then we have
	the conditions of one of the previous cases. So the remaining options are:
	they are either adjacent to each other, or they have some unobserved neighbors. If
	they are all adjacent in a triplet we obviously have three independent triplets. Let any
	two of them be adjacent (say $u$ and $v$). Then, $w$ must be adjacent to a
	unobserved vertex, $z$. $z$ must have at least two more neighbors, and if it is adjacent to any of the vertices in $\{u,v,x_{2},x_{3},y_{2},y_{3}\}$
	then the triplets are easily seen. However if $z$ is adjacent only to $x_{1}$
	and $y_{1}$ the triplets cannot be found. It is easily seen that this graph
	is a double windmill (whether $x_{1}$ and $y_{1}$ are adjacent to each other
	makes no difference) and the Claim follows for Lemma \ref{lm1}. Similarly, if
	none of $u$, $v$ and $w$ are adjacent to each other they each must have at
	least one more neighbor. If any two have the same neighbor the triplets can
	be found. Suppose that each of them has a different neighbor. Either there are three independent
	triplets or those neighbors are further adjacent only to $x_{1}$ and $y_{1}$
	and we have again obtained a double windmill, (Figure \ref{fig:c3c3_2}.), so the
	Claim follows from Lemma \ref{lm1}. $\square $
	
		\begin{figure}[h]
		\centering\includegraphics[scale=0.4]{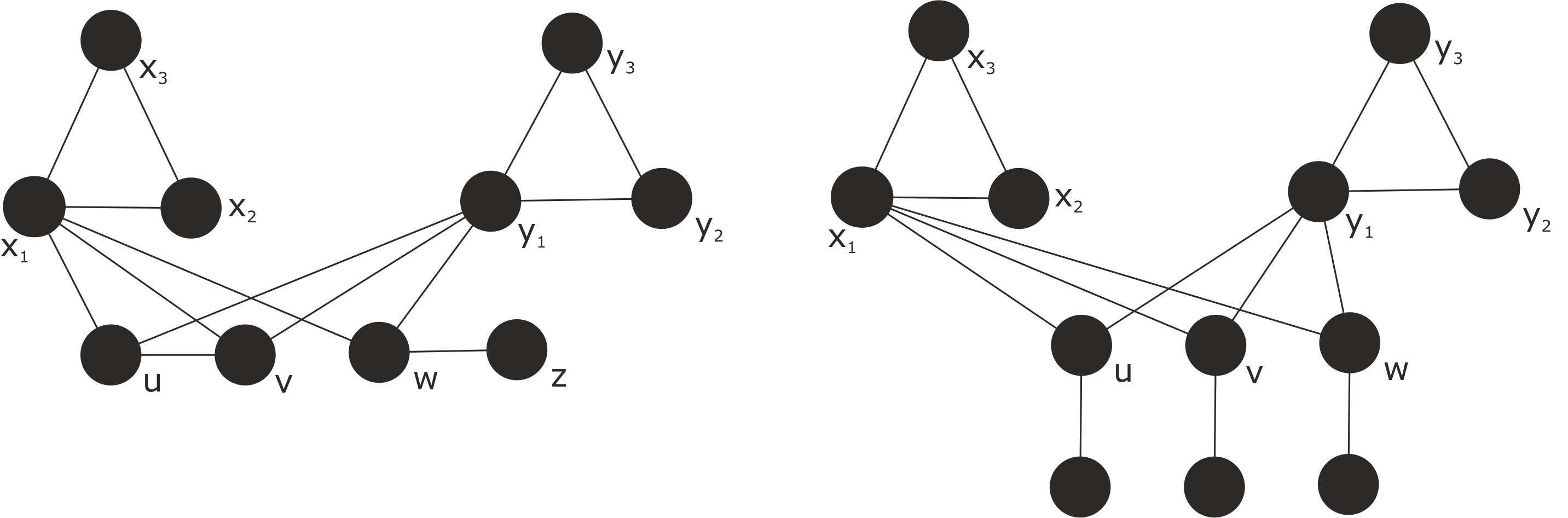}
		\caption{Examples of the second case}
		\label{fig:c3c3_2}
	\end{figure}

	\textbf{CLAIM\ 4}.\ A connected graph $G$ with $n\geq 9$ and $\delta \geq 3$
	which contains two independent $C_{4}$ cycles is safely $3$-colorable.
	
	\textit{Proof of Claim 4.}\ Besides the vertices in the cycles, $G$ has at least one
	more vertex, $u$, and as before, let us assume $u$ has at least two neighbors in $%
	C_{4}\cup C_{4}$. If those two neighbors are in different cycles then $u$ and
	those neighbors make one triplet, and the other two are the remains of
	cycles. So let us assume both neighbors of $u$ are in one of the cycles. If
	those are adjacent vertices then they form a triangle with $u$ and we have
	one $C_{3}$ and one $C_{4}$ cycle so the claim follows from Claim 2. Let us
	observe a case when $u$ is connected to two unadjacent vertices of the same $%
	C_{4}$ cycle. Since $G$ is a connected graph there must be a path between the
	two cycles so let us assume that they are connected by an edge. There are
	two options.
	
	1) The edge is incident to one of the neighbors of $u$.
	
	2) The edge in not incident to any neighbor of $u$.
	
	In both cases the three triplets are easily seen, which is illustrated in
	Figure \ref{fig:c4c4}. $\square $
	
	\begin{figure}[h]
	\centering\includegraphics[scale=0.4]{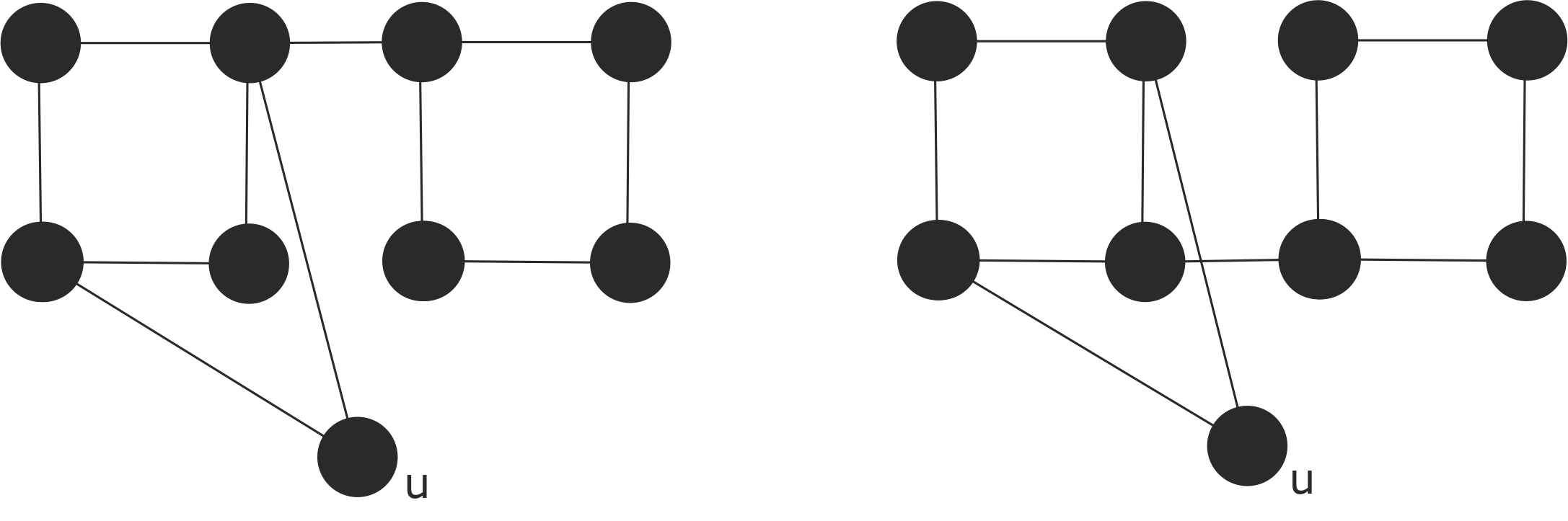}
	\caption{Two independent $C_{4}$ cycles}
	\label{fig:c4c4}
\end{figure}
	
	\textbf{CLAIM 5}.\ A connected graph $G$ with $n\geq 9$ and $\delta \geq 3$
	which contains independent $C_{4}$ and $T_{3}$ is safely $3$-colorable.
	
	\textit{Proof of Claim 5.} Obviously $G$ has at least two more vertices which are not
	in $C_{4}$ or $T_{3}$, let us denote them by $u$ and $v$, and same as in
	previous proofs let us assume they have at least two neighbors in $C_{4}\cup T_{3}$.
	If $u$ and $v$ both have at least one neighbor in $C_{4}$ the triplets can
	be found as follows:
	1) If $u$ and $v$ are neighbors with the same vertex, $x$,
	in $C_{4}$ the triplets are $uxv$, $T_{3}$ and $C_{4}\backslash \{x\}$.
	
	2) If $u$ and $v$ are neighbors with different vertices in $C_{4}$, namely $%
	x_{1}$ and $x_{2}$, then the triplets are $ux_{1}$ and neighbor of $x_{1}$
	in $C_{4}$ different from $x_{2}$, $vx_{2}$ and the remaining vertex of $%
	C_{4}$, and $T_{3}$.
	
	So let us assume at least one of $u$ and $v$ doesn't have any neighbors in $%
	C_{4}$, and let us assume it is vertex $u$. Then both neighbors of $u$ are
	in $T_{3}$ and if they are adjacent vertices we have a $C_{3}$ and a $C_{4}$
	so the claim follows from Claim 2. If they are not adjacent vertices
	then we have two $C_{4}$ cycles and the claim follows from Claim 4.$\square $
	
	\textbf{CLAIM 6}. A connected graph $G$ with $n\geq 9$ and $\delta \geq 3$
	which contains independent $C_{3}$ and $T_{3}$ is safely $3$-colorable or it is a
	double windmill.
	
	\textit{Proof of Claim 6.} It is easy to see that $G$ must have at least three more
	vertices besides $C_{3}$ and $T_{3}$. Let us denote them $u$, $v$ and $w$%
	and assume that each of them has at least two neighbors in $C_{3}\cup T_{3}$,
	following the same reasoning as in previous proofs. Also let us denote the
	vertices in $C_{3}$ by $x_{1},x_{2},x_{3}$, and vertices in $T_{3}$ by $%
	y_{1},y_{2},y_{3}$. If one of $u$, $v$ and $w$ has both neighbors in a cycle 
	$C_{3}$ then they form a $C_{4}$ cycle and the claim follows form Claim 5.
	Analogously, if any vertey of $u$, $v$ and $w$ has both neighbors in $T_{3}$
	they form a $C_{3}$ or a $C_{4}$ and the claim follows from Claim 3, or from
	Claim 2. So let us assume each of $u$, $v$ and $w$ has one neighbor in $%
	C_{3}$ and one in $T_{3}$. We consider three cases:
	
	1) Exactly two vertices from $\{u,v,w\}$ have the same neighbor in $C_{3}$,
	for instance $u$ and $v$ are adjecent with $x_{1}$, and $w$ is adjacent with 
	$x_{2}$. Then the triplets are $ux_{1}v$, $wx_{2}x_{3}$ and $y_{1}y_{2}y_{3}$%
	.
	
	2) $u$, $v$ and $w$ all have different neighbors in $C_{3}$, namely $x_{1}$, 
	$x_{2}$ and $x_{3}$ (say in this order).
	
	2.1.) $u$, $v$ and $w$ all have different neighbors in $T_{3}$, namely $%
	y_{1} $, $y_{2}$ and $y_{3}$ (say in this order). The triplets are $%
	x_{1}uy_{1}$, $x_{2}vy_{2}$ and $x_{3}wy_{3}$.
	
	2.2.) $u$ and $v$ have the same neighbor $y_{1}$ and $w$ is adjacent to $%
	y_{2}$ or $y_{3}$. Then the triplets are $uy_{1}v$, $wy_{2}y_{3}$, $%
	x_{1}x_{2}x_{3}$.
	
	2.3.) $u$ and $v$ have the same neighbor $y_{2}$ and $w$ is adjacent to $%
	y_{1}$ (or $y_{3}$). Then the triplets are $x_{2}x_{1}u$, $x_{3}wy_{1}$(or $%
	y_{3}$), $vy_{2}y_{3}$(or $y_{1}$).
	
	2.4.) All three vertices are adjacent to the same vertex in $T_{3}$. If that
	vertex is the end one, let us assume it is $y_{1}$, then we observe neighbors
	of $y_{3}$. If $y_{3}$ is adjacent to $y_{1}$ we have two $C_{3}$ cycles and
	the claim follows from the Claim 3. If $y_{3}$ is adjacent to any of $u$, $v$
	and $w$ we have a $C_{4}$ and a $C_{3}$. If $y_{3}$ has a neighbor not in
	the set $\{u,v,w,x_{1},x_{2},x_{3},y_{1}\}$ then the triplets are easily
	found, and if $y_{3}$ is adjacent to any vertex in $C_{3}$, for instance, $%
	x_{1}$, then the triplets are $y_{2}y_{3}x_{1}$, $uy_{1}v$ and $wx_{3}x_{2}$.
	(Figure \ref{fig:c3t3}.)
	
		\begin{figure}[h]
		\centering\includegraphics[scale=0.4]{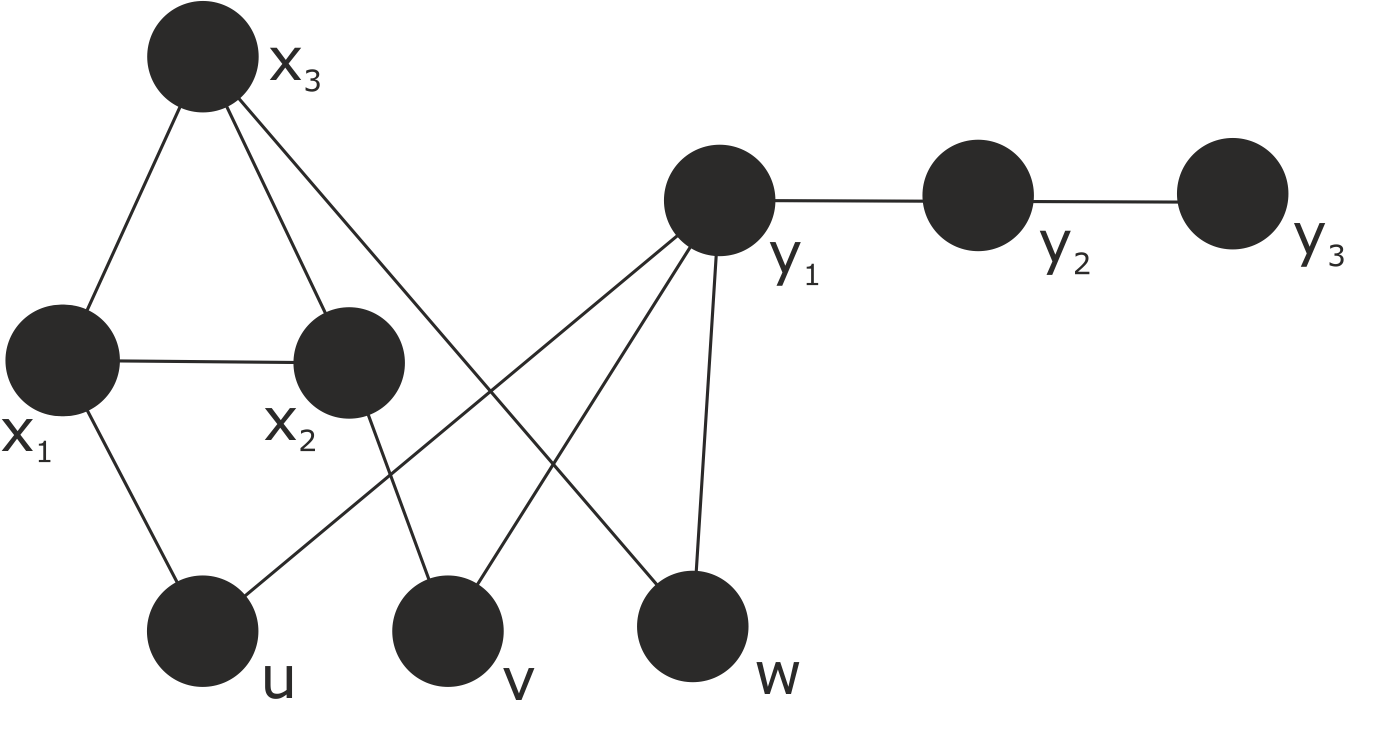}
		\caption{Independent $C_{3}$ and $T_{3}$}
		\label{fig:c3t3}
	\end{figure}
	
	On the other hand, if $u$, $v$ and $w$ are not adjacent to end vertex in $%
	T_{3}$ but instead are adjecent to $y_{2}$, again we observe possible
	neighbors of $y_{3}$ and very similarly come to the same conclusions.
	
	3) The third option is that $u$, $v$ and $w$ all have the same neighbor in $%
	C_{3}$ and let us assume it is $x_{1}$. We consider some subcases, depending
	on neighbors of $u$, $v$ and $w$ in $T_{3}$.
	
	3.1.) $u$, $v$ and $w$ all have different neighbors in $T_{3}$. Let us
	assume they are adjacent to $y_{1}$, $y_{2}$ and $y_{3}$, in this order.
	(Figure \ref{fig:c3t3_1}.) 
	
		\begin{figure}[h]
		\centering\includegraphics[scale=0.4]{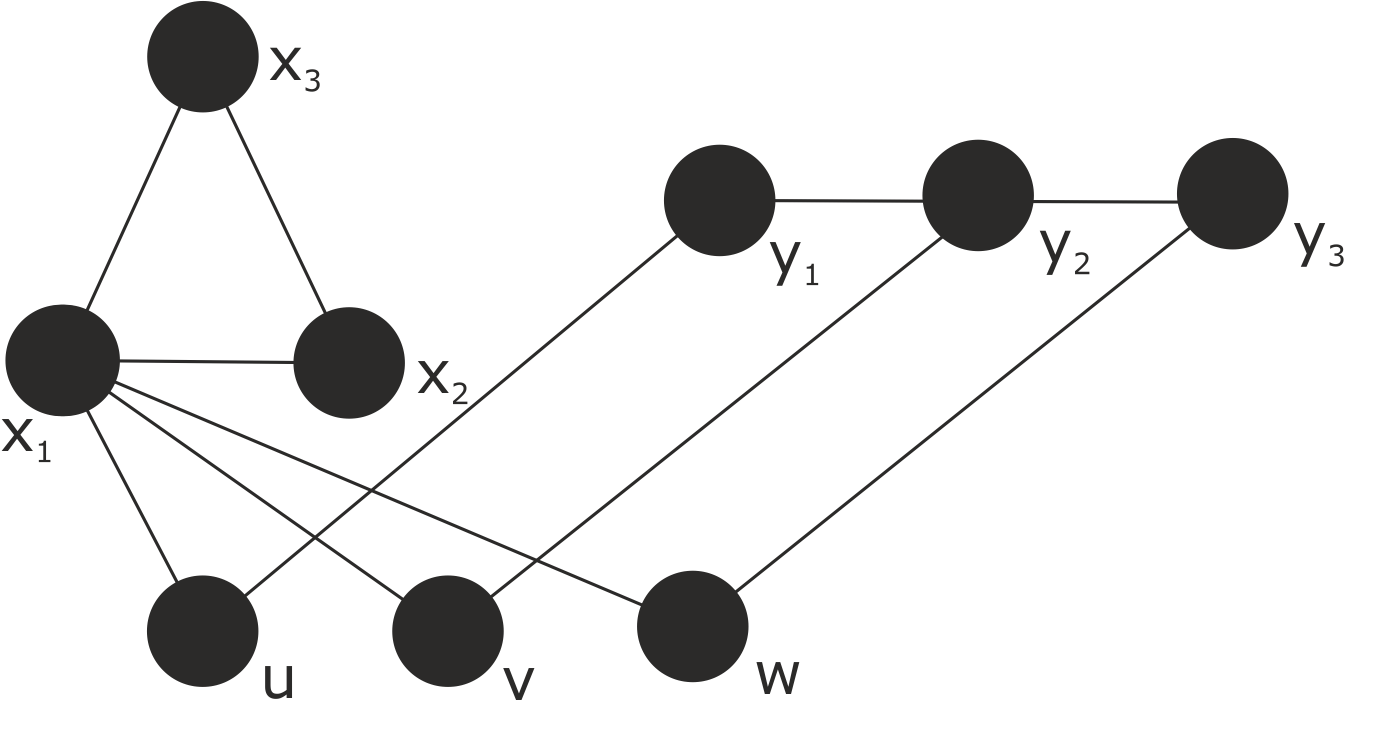}
		\caption{Subcase 3.1.}
		\label{fig:c3t3_1}
	\end{figure}
	
	We consider possible neighbors of $%
	x_{2}$. If it is adjacent to some vertex not in the set $%
	\{u,v,w,y_{1},y_{2},y_{3}\}$ then triplets can be easily found. We have
	already considered the cases in which it is adjacent to any of $\{u,v,w\}$,
	so let us assume it is adjacent to some vertex in $T_{3}$. If it is the end
	vertex, for instance $y_{1}$, then the triplets are $x_{3}x_{2}y_{1}$, $%
	ux_{1}v$ and $wy_{3}y_{2}$. And if it is adjacent to $y_{2}$ then we have to
	further consider the neighbors of $y_{1}$. Again, the triplets are easily
	found if $y_{1}$ has a previously unobserved neighbor, or if it is adjacent to $y_{3}$, $u$, $%
	v $ or $w$.
	
	If $y_{1}$ is adjacent to $x_{2}$ or to $x_{3}$ then the triplets are $%
	x_{3}x_{2}y_{1}$, $ux_{1}v$ and $wy_{3}y_{2}$.
	
	If $y_{1}$ is adjacent to $x_{1}$ then we must look at $x_{3}$. 
	It is in the same position as $x_{2}$, so all of the previous
	consideration works. And again we come to the case when also $x_{3}$ is
	adjecent to $y_{2}$. But now we have two independent $C_{3}$ cycles, $%
	x_{1}y_{1}u$, and $x_{2}y_{2}x_{3}$. (Figure \ref{fig:c3t3_2}.) 
	
		\begin{figure}[h]
		\centering\includegraphics[scale=0.4]{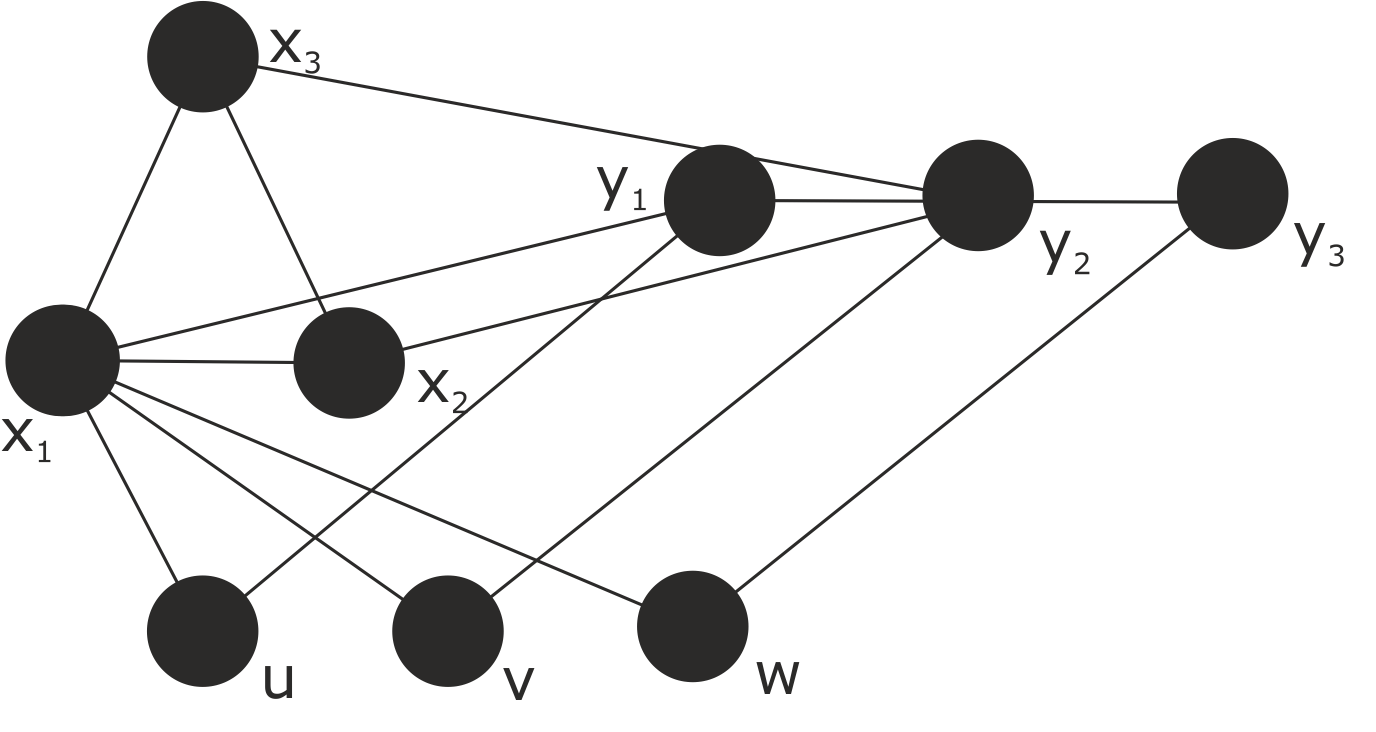}
		\caption{Two independent triangles $x_{1}y_{1}u$ and $x_{2}y_{2}x_{3}$}
		\label{fig:c3t3_2}
	\end{figure}
	
	$u$, $v$ and $w$ have exactly two different neighbors in $T_{3}$. Let us assume 
	$u$ and $v$ have the same neighbor. If it is the end vertex, for instance $%
	y_{1}$, then the vertices $ux_{1}vy_{1}$ form a $C_{4}$ and $w$, $y_{2}$ and $%
	y_{3}$ form a $T_{3}$ so the claim follows from Claim 5. On the other hand,
	if $u$ and $v$ are adjacent to $y_{2}$ and $w$ to $y_{1}$ or $y_{3}$ (say $%
	y_{1}$), then $ux_{1}vy_{2}$ form a $C_{4}$ and an independent triplet can
	be found by looking at $w$, $y_{1}$ and the remaining neighbors of $w$.
	
	3.3.) $u$, $v$ and $w$ all have the same neighbor in $T_{3}$.
	
	3.3.1) Let that neighbor be one of the end ones, say $y_{1}$. Let us look at
	the remaining neighbors of $y_{3}$.
	
	If $y_{3}$ has a previously unobserved neighbor, $z$, then we have triplets $uy_{1}v,$ $%
	x_{1}x_{2}x_{3}$ and $y_{2}y_{3}z$.
	
	If $y_{3}$ is adjacent to $y_{1}$ we have two $C_{3}$ cycles so the graph is 
	safely $3$-colorable or it is a double windmill by Claim 3.
	
	If $y_{3}$ is adjacent to $u$, $v$ or $w$ we have a $C_{4}$ and a $C_{3}$,
	so the claim follows from Claim 2.
	
	If $y_{3}$ is adjacent to at least two vertices in $C_{3}$, then $y_{3}$
	and $C_{3}$ form a $C_{4}$ cycle, and together with $uy_{1}v$ triplet we can
	apply Claim 5.
	
	3.3.2) $u$, $v$ and $w$ are adjacent to $y_{2}$.
	
	Analogous claims can be found by observing neighbors of $y_{1}$ and $y_{2}$. 
	$\square $
	
	\textbf{CLAIM 7}. A connected graph $G$ with $n\geq 9$ and $\delta \geq 3$
	which contains two independent $T_{3}$ triplets is safely $3$-colorable or it is a
	double windmill.
	
	\textit{Proof of Claim 7.} We can easily see that $G$ must have at least three more
	vertices not contained in the two triplets, let us denote them $u$, $v$ and $%
	w$. We can also assume, as before, each of them has at least two
	neighbors in the two triplets. If one of these vertices has two neighbors in
	the same triplet then they form a $C_{4}$ or a $C_{3}$ and we proceed as in
	the Claim 5 or 6. Let us assume that each of $u$, $v$ and $w$ has one
	neighbor in one $T_{3}$ and the other neighbor in the other $T_{3}$. Let us
	denote the vertices in the triplets by $x_{1},x_{2},x_{3}$ and $%
	y_{1},y_{2},y_{3}$. We distinguish four cases:
	
	1) In one of the triplets, $u$, $v$ and $w$ have exactly two neighbors.
	
	1.1.) A shared neighbor is the end one. Without the loss of generality we
	may assume $u$ and $v$ are adjacent to $x_{1}$, and $w$ is adjacent to $%
	x_{2} $ or $x_{3}$. The triplets are $ux_{1}v$, $wx_{2}x_{3}$, $%
	y_{1}y_{2}y_{3}$.
	
	1.2.) A shared neighbor is the middle one. Let us assume $u$ and $v$ are
	adjacent to $x_{2}$ and $w$ is adjacent to $x_{1}$. $x_{3}$ must have at
	least two more neighbors, and if any one of them is in the set $%
	\{u,v,w,x_{1}\}$ we have the conditions of some of the previous cases. If
	one or both of them are some unobserved vertices the triplets can be easily found.
	And if both of them are in $\{y_{1},y_{2},y_{3}\}$ then they form a $C_{3}$
	or a $C_{4}$ and together with the triplet $ux_{2}v$ we can proceed as in
	the Claim 5 or 6.
	
	2) $u$, $v$ and $w$ all have different neighbors in both triplets. Here the
	solution is easily seen, since $u$, $v$ and $w$ are centers of their
	independent triplets.
	
	3) $u$, $v$ and $w$ have the same neighbor in one triplet and all the
	different neighbors in the other triplet.
	
	3.1.) The same neighbor is the end one. Let us assume $u$, $v$ and $w$ are
	adjacent to $x_{1}$, and more, $u$ is adjacent to $y_{1}$, $v$ to $y_{2}$,
	and $w$ to $y_{3}$. $x_{3}$ must have at least two more neighbors. We have
	already considered the options when any of them are in the set $%
	\{u,v,w,x_{1}\}$ and if any of them is some unobserved vertex, the triplets are
	easily seen. But, as in the case 1.2., if $x_{3}$ is adjacent to two
	vertices of the set $\{y_{1},y_{2},y_{3}\}$ they form a $C_{3}$ or a $C_{4}$%
	, and we already have an independent triple $ux_{1}v$.
	
	3.2) The same neighbor is the middle one. This case is analogous to the case
	3.1.
	
	4) $u$, $v$ and $w$ have the same neighbor in each of the triplets.
	
	4.1.) The shared neighbor is the end one in at least one of the triplets
	(say $x_{1}$). We consider the two remaining neighbors of $x_{3}$. Either
	one of them is an unobserved vertex, in which case the triplets are easily seen, of
	both must be in $\{y_{1},y_{2},y_{3}\}$ and we find either $C_{4}$ or $C_{3}$%
	, and an independent triplet.
	
	4.2.) The shared neighbor is the middle one in both triplets, in other words 
	$u$, $v$ and $w$ are adjacent to $x_{2}$ and $y_{2}$. Let us observe
	neighbors of $x_{1}$. If it is adjacent to $x_{3}$ or any of $\{u,v,w\}$
	than we have a $C_{3}$ and a $T_{3}$ and the claim follows as the Claim 6.
	If $x_{1}$ has two new neighbors than the triplets are easily seen. If $%
	x_{1} $ has exactly one unobserved neighbor, namely $x$, and is adjacent to $y_{1}$
	or $y_{3}$ (let us assume $y_{1}$), then the triplets are $xx_{1}y_{1}$, $%
	ux_{2}v$, $wy_{2}y_{3}$. And if $x_{1}$ has exactly one unobserved neighbor, namely 
	$x$, and is adjacent to $y_{2}$, then we turn to the neighbors of $x$. If $x$
	is also adjacent to any vertex in $\{x_{2},x_{3},u,v,w,y_{1},y_{2},y_{3}\}$
	we have a $C_{3}$ or a $C_{4}$ and a triple $T_{3}$, and if it has two new
	neighbors then the independent triplets are easily seen. (Figure \ref{fig:t3t3}).
	$\square $
	
		\begin{figure}[h]
		\centering\includegraphics[scale=0.4]{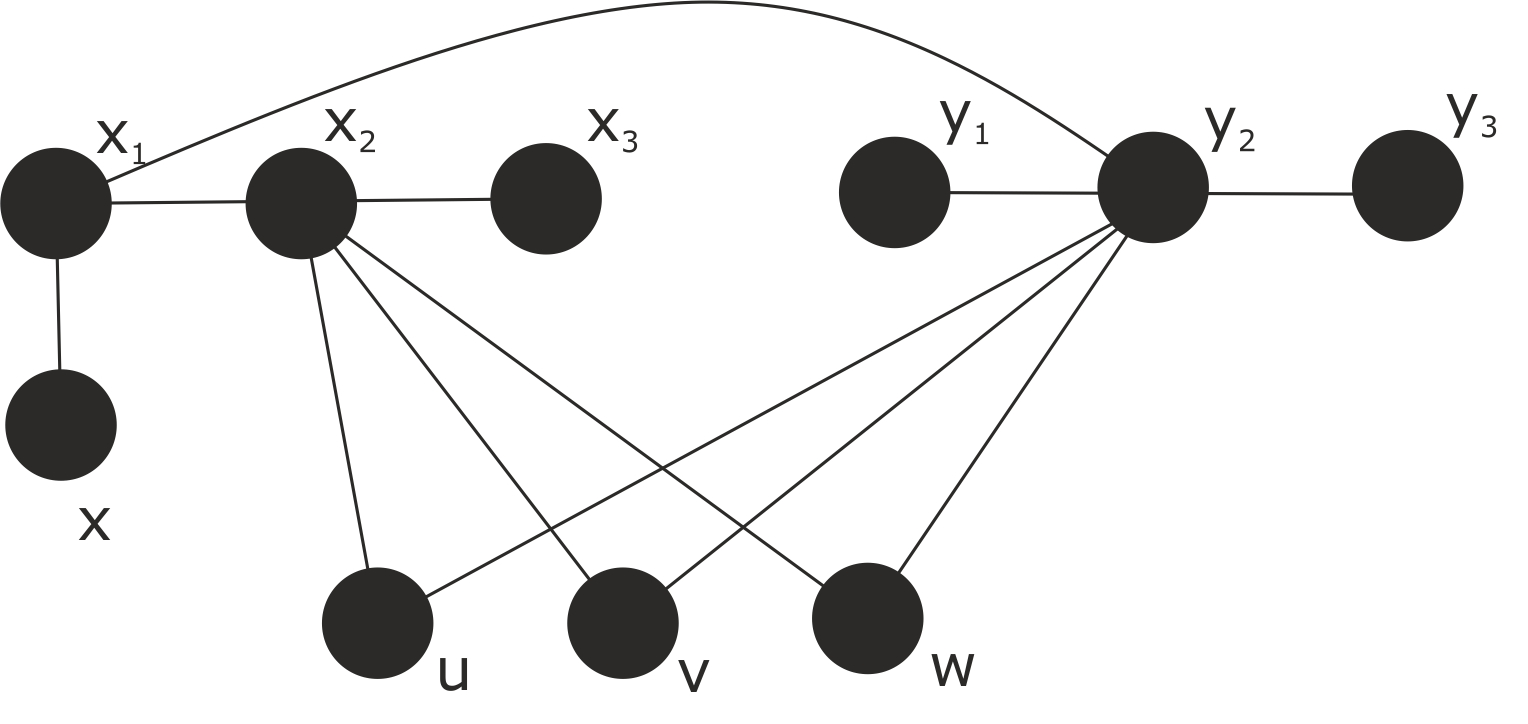}
		\caption{A subcase of 4.2.}
		\label{fig:t3t3}
	\end{figure}
	
	Now the original claim follows from Claims 1-7 and Proposition \ref{prop1}.
\end{proof}

This concludes our observations for connected graphs and we now turn to
graphs with more than one component. Let $G$ be a graph with $\delta (G)\geq
3$. It can be easily seen that:

i) $G$ has at least four vertices in each component.

ii) If $G$ has $3$ or more components, it contains at least three
independent triples so it is safely $3$-colorable.

iii) If $G$ has exactly two components at least one must be safely $3$-colorable or
both must be $1$-safely $3$-colorable in order for $G$ to be safely $3$-colorable.

iv) For a graph to be $1$-safely $3$-colorable it has to have at least $6$ vertices,
because each color must be present at at least two vertices.

v) Graph is $1$-safely $3$-colorable if it has at least two independent triplets.

\begin{lemma}
	A connected graph $G$ with $n\geq 6$ vertices and $\delta \geq 3$ has at least two
	independent triplets.
\end{lemma}

\begin{proof}
	Note that $G$ has at least one connected triplet $T_{3}$. Let us denote its
	vertices by $u$, $v$ and $w$. Let us distinguish two cases:
	
	1) $T_{3}$ is not a triangle.
	
	Let $v$ be the middle vertex. $u$ and $w$ have at least two more neighbors.
	Let $x$ and $y$ be the neighbors of $u$. If $w$ has at least one neighbor
	different from $x$ and $y$, let us denote it by $z$, then the triplets are $%
	xuy$ and $zwv$. And if $w$ is also adjacent to $x$ and $y$ and to none else,
	than we observe vertices $x$, $y$ and $v$. They all must have at least one
	more neighbor because of $\delta \geq 3$ and at least one of them must have
	a neighbor not contained in $\{u,v,w,x,y\}$ because of $n\geq 6$. Without
	the loss of generality let us assume that $x$ has a neighbor $z$. Now the
	triplets are $zxu$ and $ywv$.
	
	2) $T_{3}$ is a triangle.
	
	All of $u$, $v$ and $w$ must have at least one more neighbor.
	
	2.1.) They all have the same neighbor, namely $x$. Now we have a $K_{4}$
	graph, and obviously we are missing at least two more vertices because of $%
	n\geq 6$. Whether the two vertices are adjacent to the same vertex in $K_{4}$, or not, the two triplets are easily seen.
	
	2.2.) Two of the vertices, let us assume $u$ and $v$, have the same neighbor, $x$,
	and $w$ has a neighbor $y$. If $x$ and $y$ have no neighbors not contained
	in $\{u,v,w\}$ than at least one of $u$, $v$ and $w$ must have another
	neighbor because of $n\geq 6$. But in that case, the same as in the first case, $%
	x $ or $y$ having another neighbor, the triplets are easily seen.
\end{proof}

Now we have proven the Theorem form the beginning:

\begin{theorem}
	Graph $G$ with $\delta \geq 3$ is safely $3$-colorable if at least one of the
	following stands:
	
	i) $G$ has at least three components;
	
	ii) $G$ has two components with at least $6$ vertices each;
	
	iii) $G$ has at least one component with at least $9$ vertices which is
	different from a double windmill.
	
\end{theorem}

\begin{proof}
	The proof follows from previous observations.\bigskip
\end{proof}

\section{Additional results}

We now give some additional results we acquired. First one is regarding the time complexity of the algorithm that checks if a graph that is $3$-colored is safely $3$-colored. 

\begin{remark}
	For any given graph $G$ with $n$ vertices and $m$ edges, and a $3$-coloring $\phi $ of vertices in $G$, an algorithm exists which can
	determine if the coloring is safe $3$-coloring, in $O(n^{2}m)$ time complexity.
	We have to check all pairs of vertices, remove them from $G$ and
	then browse the remaining vertices of $G$ by components, checking if we
	have all $3$ colors. When we find them we move to the next pair of
	vertices, and so on. If we find a pair of vertices such that with their removal there doesn't exist a component with all $3$ colors, then the coloring is not safe. To check all the pairs of
	vertices we need $O(n^{2})$ time (there are $\binom{n}{2}=\frac{n^{2}-n}{2}$ pairs of vertices), and for each pair removed we check the
	remaining components for colors in $O(m)$ time (we check each component by the edges, looking at the colors of those vertices).
\end{remark}

The second remark concerns time complexity of algorithm that checks if a graph has $3$ independent triplets. If it does, then it is safely $3$-colorable. However, the converse is not true. There are safely $3$-colorable graphs that don't contain $3$ independent triplets which is presented in Remark \ref{obrat}.

\begin{remark}\label{alg}
	For any given graph $G$ with $n$ vertices, it takes $O(n^{4})$ time to check if it contains
	three independent triplets. Algorithm is based on checking all $3$-subsets
	of $V(G)$ and determining if those three vertices may be the centers of
	independent triplets. Let us denote the three chosen vertices by $a$, $b$
	and $c$, and by\\
	$s_{A}$ - the set of neighbors of vertex a, not counting
	vertices b and c;\\ 
	$s_{B}$ - the set of neighbors of vertex b, not counting
	vertices and c; \\
	$s_{C}$- the set of neighbors of vertex c, not counting
	vertices b and a; \\
	$s_{AB}$ - the set of neighbors of vertices a and b, not
	counting vertices a, b, c; \\
	$s_{BC}$ - the set of neighbors of vertices b and c, not
	counting vertices a, b, c; \\
	$s_{AC}$ - the set of neighbors of vertices a and c, not
	counting vertices a, b, c; \\
	$s_{ABC}$ - the set of neighbors of vertices a, b and c,
	not counting vertices a, b, c.\\\medskip
	
	Also, let us denote by $n_{X}$ the cardinal number of the set $s_{X}$, $$
	X=A,B,C,AB,AC,BC,ABC.$$ By examining these cardinal numbers we can determine
	if $a$, $b$ and $c$ can be centers of independent triplets. The way of doing
	this is given in the following theorem. Checking all $3$-subsets of vertices
	takes $O(n^{3})$ time, and calculating the equations from Theorem \ref{abc} takes $O(n)
	$ time.
\end{remark}

\begin{theorem}\label{abc}
	Let $a$, $b$ and $c$ be any three different vertices in the graph, with the notation from Remark \ref{alg}. Vertices $%
	a$, $b$ and $c$ are the centers of independent triplets if and only if the
	following holds
	\begin{center}
			$n_{A},n_{B},n_{C} \geq 2;$ \\
		$n_{A}+n_{B}-n_{AB} \geq 4;$ \\
		$n_{A}+n_{C}-n_{AC} \geq 4;$ \\
		$n_{B}+n_{C}-n_{BC} \geq 4;$ \\
		$n_{A}+n_{B}+n_{C}-n_{AB}-n_{AC}-n_{BC}+n_{ABC} \geq 6.$
	\end{center}
\end{theorem}

\begin{proof}
	First let us prove the analogous claim, for two vertices.
	
	\textbf{Claim 1.} Vertices $a$ and $b$ are centers of independent triplets if and
	only if the following holds%
	$$
	n_{A},n_{B} \geq 2;$$
	$$n_{A}+n_{B}-n_{AB} \geq 4,
	$$
	where $n_{A}$ is the number of neighbors of vertex $a$, not counting vertex $%
	b$, $n_{B}$ is the number of neighbors of $b$ not counting vertex $a$, and $%
	n_{AB}$ is the number of vertices adjacent to $a$ and $b$, not counting
	themselves, if they are adjacent. If $a$ and $b$ are the centers of
	independent triplets, it is easy to see that the given inequalities hold.
	Let us prove the converse. We now denote 
	\begin{center}
			$s_{a} =s_{A}\backslash s_{AB}$ \\
		$s_{b} =s_{B}\backslash s_{AB}$ \\
		$s_{ab} =s_{AB}.$
	\end{center}

	It holds%
	\begin{center}
		$n_{a} =n_{A}-n_{AB}$ \\
	$n_{b} =n_{B}-n_{AB}$ \\
	$n_{ab} =n_{AB},$
	\end{center}
	
	so the inequalities are equivalent to 
	
	\begin{center}
	$n_{a}+n_{ab} \geq 2$ \\
$n_{b}+n_{ab} \geq 2$ \\
$n_{a}+n_{b}+n_{ab} \geq 4.$
\end{center}

	We consider three cases.
	
	1) If $n_{a}\geq 2$ then those two neighbors of $a$ form the $a$-centered
	triplet, and from $n_{b}+n_{ab}\geq 2$ we see that $b$ has at least two
	neighbors which can form another triplet.
	
	2) If $n_{a}=1$ then $n_{ab}\geq 1$, and from $n_{a}+n_{b}+n_{ab}\geq 4$ we
	conclude that one of the following stands:%
		\begin{center}
		$n_{b} =0$ and $n_{ab}\geq 3$ \\
	$n_{b} =1$ and $n_{ab}\geq 2$ \\
	$n_{b} =2$ and $n_{ab}\geq 1.$
	\end{center}

	It is easily seen that in all cases we have the two triplets.
	
	3) If $n_{a}=0$ then $n_{ab}\geq 2$. Analogously as in 2) we can see that
	one of the following stands:%
	\begin{center}
	$n_{b} =2$ and $n_{ab}\geq 2$ \\
	$n_{b} =1$ and $n_{ab}\geq 3$ \\
	$n_{b} =0$ and $n_{ab}\geq 4$
	\end{center}

	so analogously we conclude that the triplets are easily found. This
	concludes proof of Claim 1.
	
	Let us now prove the original claim, for three vertices, $a$, $b$ and $c$.
	One direction can be easily seen, that is if $a$, $b$ and $c$ are the
	centers of independent triplets, then the inequalities hold. Let us prove
	the converse. First we denote:
		\begin{center}
		$s_{a} =s_{A}\backslash (s_{AB}\cup s_{AC})$ \\
		$s_{b} =s_{B}\backslash (s_{AB}\cup s_{BC})$ \\
		$s_{c} =s_{C}\backslash (s_{AC}\cup s_{BC})$ \\
		$s_{ab} =s_{AB}\backslash s_{ABC}$ \\
		$s_{ac} =s_{AC}\backslash s_{ABC}$ \\
		$s_{bc} =s_{BC}\backslash s_{ABC}.$
	\end{center}

	It holds%
		\begin{center}
		$n_{a} =n_{A}-n_{AB}-n_{AC}+n_{ABC}$ \\
		$n_{b} =n_{B}-n_{AB}-n_{BC}+n_{ABC}$ \\
		$n_{c} =n_{C}-n_{AC}-n_{BC}+n_{ABC}$ \\
		$n_{ab} =n_{AB}-n_{ABC}$ \\
		$n_{ac} =n_{AC}-n_{ABC}$ \\
		$n_{bc} =n_{BC}-n_{ABC}$ \\
		$n_{abc} =n_{ABC},$
	\end{center}

	so the inequalities are equivalent to%
		\begin{center}
			$n_{a}+n_{ac}+n_{ab}+n_{abc} \geq 2$ \\
		$n_{b}+n_{bc}+n_{ab}+n_{abc} \geq 2$ \\
		$n_{c}+n_{ac}+n_{bc}+n_{abc} \geq 2$ \\
		$n_{a}+n_{b}+n_{c}+n_{ab}+n_{bc}+n_{abc} \geq 4$ \\
		$n_{a}+n_{c}+n_{ac}+n_{ab}+n_{bc}+n_{abc} \geq 4$ \\
		$n_{b}+n_{c}+n_{ac}+n_{ab}+n_{bc}+n_{abc} \geq 4$ \\
		$n_{a}+n_{b}+n_{c}+n_{ac}+n_{ab}+n_{bc}+n_{abc} \geq 6.$
	\end{center}

	We will observe three cases, depending on values of $n_{a}$, $n_{b}$ and $%
	n_{c}$.
	
	\textbf{Case 1.} If any of $n_{a}$, $n_{b}$, $n_{c}$ is at least $2$, without the
	loss of generality let us assume $n_{c}\geq 2$, than the claim follows from
	Claim 1. Indeed, since $n_{c}\geq 2$, we can use those two neighbors of $c$
	to form a $c$-centered triplet, and if we denote%
		\begin{center}
		$n_{a}+n_{ac} =n_{a}^{\prime }$ \\
		$n_{b}+n_{bc} =n_{b}^{\prime }$ \\
		$n_{ab}+n_{abc} =(n_{ab})^{\prime },$
	\end{center}

	from the inequalities we have%
	\begin{center}
	$n_{a}^{\prime }+n_{b}^{\prime }+(n_{ab})^{\prime } \geq 4$ \\
	$n_{a}^{\prime }+(n_{ab})^{\prime } \geq 2$ \\
	$n_{b}^{\prime }+(n_{ab})^{\prime } \geq 2,$
	\end{center}

	and Claim 1 can be directly applied.
	
	\textbf{Case 2.} $n_{a}=n_{b}=n_{c}=0$. Than we have $n_{ac}+n_{ab}+n_{bc}+n_{abc}%
	\geq 6$. Let us consider some subcases.
	
	2.1. $n_{abc}\geq 6$. We have at least $6$ vertices adjacent to each of $a$, 
	$b$ and $c$ so obviously in this case we can use $2$ of the vertices for
	each of the triplets.
	
	2.2. $n_{abc}=5$. Now at least one of $n_{ac}$, $n_{ab}$ and $n_{bc}$ must
	be at least $1$. Let us assume $n_{ab}\geq 1$. Then for the $a$-centered
	triplet we use one vertex from $s_{ab}$, and one from $s_{abc}$, and for the
	other two triplets we use $4$ of the remaining vertices from $s_{abc}$.
	
	2.3. $n_{abc}=4$. In this case from the inequality follows that at least
	one of the $n_{ab}$, $n_{bc}$, $n_{ac}$ is at least $2$, or at least two of
	them are at least $1$. If one of them, let us assume $n_{ab}$, is at least $2$%
	, than we can use those $2$ vertices for the $a$-centered triplet, and $4$
	vertices in $s_{abc}$ for $b$ and $c$-centered triplets. And if any two of $%
	n_{ab}$, $n_{ac}$ or $n_{bc}$ are at least $1$, let us assume $n_{ab}$, $%
	n_{ac}\geq 1$, we can use those two vertices for the same triplet, in this
	case it will be $a$-centered triplet, and the rest of the triplets we form
	with vertices from $s_{abc}$.
	
	2.4. $n_{abc}=3$. It follows $n_{ac}+n_{ab}+n_{bc}\geq 3$. Now, if any of $%
	n_{ac}$, $n_{ab}$, $n_{bc}$ is at least $2$, we use those $2$ as in case
	2.3., for one triplet, and the remaining $1$ we have in these three sets
	together with $3$ from $s_{abc}$ we use for other two triplets. If neither
	of $n_{ab}$, $n_{ac}$, $n_{bc}$ is at least $2$, it means $%
	n_{ab},n_{ac},n_{bc}\geq 1$ and we use those three vertices one in each of
	the triplets and then complete the triplets with three vertices from $s_{abc}
	$.
	
	For the remaining cases consideration is very similar to case 2.4. so we
	skip the detailed proof.
	
	\textbf{Case 3.} At least one od $n_{a}$, $n_{b}$, $n_{c}$ is at least $1$ and none
	of them equals $2$. This case includes options of all three of them equaling 
	$1$, two of them equaling $1$ and one of them $0$, and one of them equaling $%
	1$ and two of them equaling $0$.
	
	3.1. $n_{a}=n_{b}=n_{c}=1$. Obviously each of the $a$, $b$ and $c$ are
	missing one more neighbor to form a triplet. From the inequalities we have $%
	n_{ac}+n_{ab}+n_{bc}+n_{abc}\geq 3$.
	
	If $n_{abc}\geq 3$ then we use one of those vertices for each of the
	triplets.
	
	If $n_{abc}=2$ then at least one of $n_{ab}$, $n_{ac}$, $n_{bc}$ is at least 
	$1$, so we use that one to complete one triplet, and the two form $s_{abc}$
	for other two triplets.
	
	If $n_{abc}=1$ then either one of $n_{ac}$, $n_{ab}$, $n_{bc}$ is at least $2
	$, or at least two of $n_{ab}$, $n_{ac}$ or $n_{bc}$ are at least $1$. Let
	us assume $n_{ab}\geq 2$. Then we use those two vertices to complete the
	triplets of $a$ and $b$ and we use vertex from $s_{abc}$ for the $c$%
	-centered triplet. And if two of $n_{ab}$, $n_{ac}$ or $n_{bc}$ are at least $1$, we
	also use is to complete two different triplets and complete the remaining
	triplet with vertex from $s_{abc}$.
	
	If $n_{abc}=0$ then among $s_{ab}$, $s_{ac}$ or $s_{bc}$ we have at least $3$
	vertices and it is important to observe that two of those sets cannot be
	equal to $0$ at the same time. For instance, if $n_{ac}=n_{ab}=0$ and $%
	n_{bc}\geq 3$, then we have a contradiction with $%
	n_{a}+n_{ac}+n_{ab}+n_{abc}\geq 2$, because $n_{a}=1$. So at least two of $%
	n_{ab}$, $n_{ac}$, $n_{bc}$ are at least $1$ and we can complete the
	triplets as before.
	
	3.2. Two of $n_{a}$, $n_{b}$, $n_{c}$ equal $1$ and one of them equals $0$.
	Without the loss of generality let us assume $n_{a}=0$, $n_{b}=n_{c}=1$. We
	now have $n_{ac}+n_{bc}+n_{ab}+n_{abc}\geq 4$. We proceed considering the
	subcases for $n_{abc}\geq 4$, and $n_{abc}\in \{0,1,2,3\}$ and consideration
	is completely analogous as in 3.1.
	
	3.3. Two of $n_{a}$, $n_{b}$, $n_{c}$ equal $0$ and one of them equals $1$.
	Without the loss of generality let us assume $n_{a}=n_{b}=0$, $n_{c}=1$. Now
	from the inequalities it follows $n_{ac}+n_{bc}+n_{ab}+n_{abc}\geq 5$ and
	the rest of the proof again easily follows by considering different values
	of $n_{abc}$.
	
	This completes the proof for conditions for having the three independent
	triplets.
\end{proof}

\begin{remark}\label{obrat}
	In Proposition \ref{prop1} we proved that a graph with given conditions is safely $%
	3$-colorable if it contains three independent triplets. However the converse
	is not true. There exists a graph which is safely $3$-colorable, but doesn't
	have three independent triplets and it can be seen in Figure \ref{fig:2sig}.%
	
	\begin{figure}[h]
		\centering\includegraphics[scale=0.6]{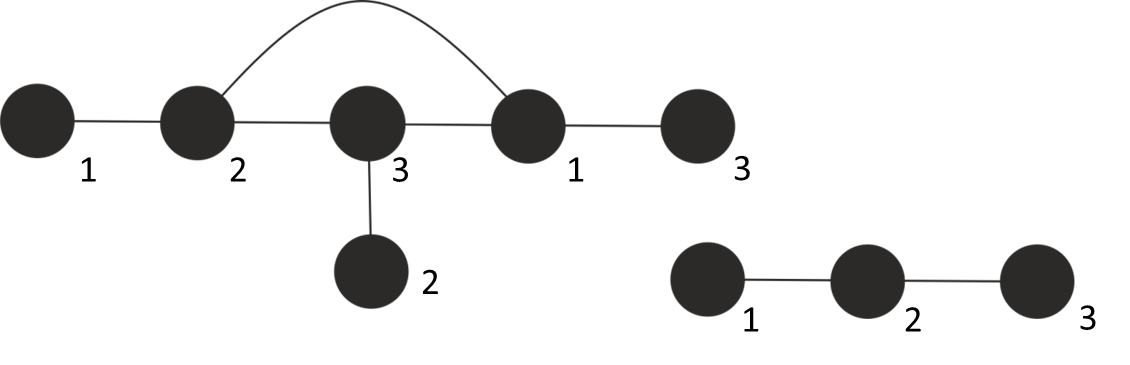}
		\caption{A safely $3$-colored graph that doesn't contain $3$ independent triplets}
		\label{fig:2sig}
	\end{figure}

\end{remark}

\section*{Conclusions and further work}
We have defined a new type of vertex coloring, an $a$-safe $k$-coloring, motivated by a secret sharing scheme and $a$ attackers that are trying to steal the secret. We limited ourselves to the assumption that the secret is divided in $3$ parts and therefore we explored what are the families of graphs with minimal degree $3$ that allow a $2$-safe coloring with $3$ colors. Further work may include exploring families of safely $k$-colorable graphs with $k>3$, or determining minimal number of colors in order to safely color any given graph family.

\section{Acknowledgements}
Partial support of the Croatian Ministry of Science and Education is gratefully acknowledged.

\end{document}